\documentclass[epsfig,latexsym,amsfonts,twoside]{article}
\usepackage{amssymb,enumerate,amsmath,amscd,graphicx}
\usepackage{pdfpages}

\usepackage{tikz}
\tikzset{help lines/.style={color=blue!50,very thin}}

\def\part#1{\frac{\partial\phantom{#1}}{\partial#1}}
\newtheorem{thm}{Theorem}

\newtheorem{proposition}[thm]{Proposition}
\newtheorem{lemma}[thm]{Lemma}
\newtheorem{corollary}[thm]{Corollary}
\newtheorem{conjecture}[thm]{Conjecture}

\newenvironment{proof}{\begin{trivlist}\item[]{\bf Proof} }%
{\hfill $\Box$ \end{trivlist}}
\newenvironment{definition}{\begin{trivlist}\item[]{\bf Definition}\em }%
{\end{trivlist}}
\newenvironment{remark}{\begin{trivlist}\item[]{\bf Remark} }%
{\end{trivlist}}
\newenvironment{example}{\begin{trivlist}\item[]{\bf Example} }%
{\end{trivlist}}
{\end{trivlist}}

\def\Z{\ifmmode{{\mathbb Z}}\else{${\mathbb Z}$}\fi}
\def\Q{\ifmmode{{\mathbb Q}}\else{${\mathbb Q}$}\fi}
\def\C{\ifmmode{{\mathbb C}}\else{${\mathbb C}$}\fi} 
\def\P{\ifmmode{{\mathbb P}}\else{${\mathbb P}$}\fi} 

\def\H{\ifmmode{{\mathrm H}}\else{${\mathrm H}$}\fi} 

\def\B{\ifmmode{{\cal B}}\else{${\cal B}$}\fi} 
\def\E{\ifmmode{{\cal E}}\else{${\cal E}$}\fi} 
\def\F{\ifmmode{{\cal F}}\else{${\cal F}$}\fi} 
\def\K{\ifmmode{{\cal K}}\else{${\cal K}$}\fi} 
\def\L{\ifmmode{{\cal L}}\else{${\cal L}$}\fi} 
\def\M{\ifmmode{{\cal M}}\else{${\cal M}$}\fi} 
\def\N{\ifmmode{{\cal N}}\else{${\cal N}$}\fi} 
\def\O{\ifmmode{{\cal O}}\else{${\cal O}$}\fi} 
\def\U{\ifmmode{{\cal U}}\else{${\cal U}$}\fi}
\def\X{\ifmmode{{\cal X}}\else{${\cal X}$}\fi} 

\def\Br{\ifmmode{{\mathrm{Br}}}\else{${\mathrm{Br}}$}\fi} 
\def\OG{\ifmmode{\widetilde{\cal M}_4}\else{$\widetilde{\cal M}_4$}\fi} 
\def\D{\ifmmode{{\cal{D}}_{\mathrm{coh}}^b}\else{${{\cal{D}}_{\mathrm{coh}}^b}$}\fi}
\def\Shah{\ifmmode{\amalg\hspace*{-3.5pt}\amalg}\else{$\amalg\hspace*{-3.5pt}\amalg$}\fi}

\newcommand{\twowheel}{\raisebox{0pt}{
                 \begin{picture}(12,12)(-6,-4)
                 \put(0,0){\circle{6}} \put(0,3){\line(0,1){4}} \put(0,-3){\line(0,-1){4}}
                 \end{picture}}}

\newcommand{\twoVgraph}{\raisebox{0pt}{
                 \begin{picture}(18,18)(-9,-5)
                 \put(0,0){\circle{16}} \put(-8,0){\line(1,0){16}}
                 \end{picture}}}

\newcommand{\fourVgraph}{\raisebox{0pt}{
                 \begin{picture}(18,26)(-9,-9)
                 \put(0,0){\oval(16,24)} \put(-8,4){\line(1,0){16}}
                 \put(-8,-4){\line(1,0){16}}
                 \end{picture}}}

\newcommand{\sixVgraph}{\raisebox{0pt}{
                 \begin{picture}(26,24)(-13,-8)
                 \put(-9,8){\circle{6}} \put(9,8){\circle{6}}
                 \put(-6,8){\line(1,0){12}} \put(0,-8){\circle{6}} 
                 \put(-9,5){\line(2,-3){7}} \put(9,5){\line(-2,-3){7}}
                 \end{picture}}}

\newcommand{\eightVgraphI}{\raisebox{0pt}{
                 \begin{picture}(26,26)(-13,-9)
                 \put(-9,9){\circle{6}} \put(9,9){\circle{6}}
                 \put(-9,-9){\circle{6}} \put(9,-9){\circle{6}} 
                 \put(-6,9){\line(1,0){12}}
                 \put(-9,6){\line(0,-1){12}}
                 \put(-6,-9){\line(1,0){12}}
                 \put(9,6){\line(0,-1){12}}
                 \end{picture}}}

\newcommand{\eightVgraphII}{\raisebox{0pt}{
                 \begin{picture}(28,28)(-14,-10)
                 \put(-13,13){\line(1,0){26}}
                 \put(-13,-13){\line(1,0){26}}
                 \put(-13,-13){\line(0,1){26}}
                 \put(13,-13){\line(0,1){26}}
                 \put(-3,3){\line(1,0){6}}
                 \put(-3,-3){\line(1,0){6}}
                 \put(-3,-3){\line(0,1){6}}
                 \put(3,-3){\line(0,1){6}}
                 \put(-13,13){\line(1,-1){10}}
                 \put(-13,-13){\line(1,1){10}}
                 \put(13,13){\line(-1,-1){10}}
                 \put(13,-13){\line(-1,1){10}}
                 \end{picture}}}

\newcommand{\tenVgraphI}{\raisebox{0pt}{
                 \begin{picture}(26,29)(-13,-9)
                 \put(-9,9){\circle{6}} \put(9,9){\circle{6}}
                 \put(-9,-9){\circle{6}} \put(9,-9){\circle{6}}
                 \put(0,9){\circle{6}}
                 \put(-6,9){\line(1,0){3}}
                 \put(6,9){\line(-1,0){3}}
                 \put(-9,6){\line(0,-1){12}}
                 \put(-6,-9){\line(1,0){12}}
                 \put(9,6){\line(0,-1){12}}
                 \end{picture}}}

\newcommand{\twoWgraph}{\raisebox{0pt}{
                 \begin{picture}(22,22)(-9,-6)
                 \put(0,3){\circle{18}} \put(0,-3){\circle{18}}
                 \end{picture}}}

\newcommand{\threeWgraph}{\raisebox{0pt}{
                 \begin{picture}(22,22)(-9,-6)
                 \put(0,0){\circle{20}} \put(-8.5,-5){\line(1,0){17}}
                 \put(0,10){\line(-3,-5){9}} \put(0,10){\line(3,-5){9}}
                 \end{picture}}}

\newcommand{\fourWgraphI}{\raisebox{0pt}{
                 \begin{picture}(22,22)(-9,-6)
                 \put(0,0){\oval(16,28)} \put(0,0){\oval(28,16)}
                 \end{picture}}}

\newcommand{\fourWgraphII}{\raisebox{0pt}{
                 \begin{picture}(22,22)(-9,-6)
                 \put(0,0){\oval(16,28)} 
                 \put(-8,8){\line(1,0){16}} \put(-8,-8){\line(1,0){16}}
                 \put(-8,8){\line(1,-1){16}} \put(-8,-8){\line(1,1){16}}
                 \end{picture}}}

\newcommand{\fourWgraphIII}{\raisebox{0pt}{
                 \begin{picture}(22,22)(-9,-6)
                 \put(-8,9){\line(0,-1){18}} \put(8,9){\line(0,-1){18}}
                 \put(-8,9){\line(1,0){16}} \put(-8,-9){\line(1,0){16}}
                 \put(0,9){\circle{16}} \put(0,-9){\circle{16}}
                 \end{picture}}}

\begin{document}

\title{Topological bounds on hyperk{\"a}hler manifolds} 
\author{Justin Sawon}
\date{December, 2021}
\maketitle

\begin{abstract}
We conjecture that certain curvature invariants of compact hyperk{\"a}hler manifolds are positive/negative. We prove the conjecture in complex dimension four, give an ``experimental proof'' in higher dimensions, and verify it for all known hyperk{\"a}hler manifolds up to dimension eight. As an application, we show that our conjecture leads to a bound on the second Betti number in all dimensions.
\end{abstract}

\section{Introduction}

In~\cite{rw97} Rozansky and Witten introduced new invariants of hyperk{\"a}hler manifolds. These invariants are defined by taking copies of the curvature tensor and the holomorphic symplectic form of the hyperk{\"a}hler manifold, and contracting indices to get a form that can be integrated over the manifold, giving an $\mathbb{R}$-valued invariant. The pattern by which one contracts the indices is encoded in a trivalent graph $\Gamma$, and the corresponding invariant is written $b_{\Gamma}$.

These Rozansky-Witten invariants were studied by the author in~\cite{sawon00}. Let $M$ be a hyperk{\"a}hler manifold of complex dimension $2n$ (in this article we will always assume our hyperk{\"a}hler manifolds are compact and irreducible, i.e., simply connected and with $h^{2,0}=1$; and by dimension we will always mean the complex dimension). Let $\Theta^n$ be the trivalent graph consisting of $n$ disjoint copies of the trivalent graph
$$\twoVgraph$$
with two vertices. Together with Hitchin~\cite{sawon00,hs01}, the author proved that
$$b_{\Theta^n}(M)=\frac{n!}{(4\pi^2n)^n}\frac{\|R\|^{2n}}{(\mathrm{vol}M)^{n-1}},$$
where $\|R\|^2$ is the $\mathcal{L}^2$-norm of the curvature and $\mathrm{vol}M$ is the volume of $M$. In particular, $b_{\Theta^n}(M)$ is positive. More recent investigations have led to the following conjecture. Let $\Theta_k$ denote the trivalent graph with $2k$ vertices consisting of a cycle of $k$ small loops; thus $\Theta_1=\Theta$ and 
$$\Theta_2=\fourVgraph\qquad\qquad\Theta_3=\sixVgraph\qquad\qquad\Theta_4=\eightVgraphI\qquad\qquad\Theta_5=\tenVgraphI\qquad\qquad\mathrm{etc.}$$

\begin{conjecture}
Let $M$ be a hyperk{\"a}hler manifold of complex dimension $2n$. Then for $2\leq k\leq n$ the Rozansky-Witten invariant $b_{\Theta^{n-k}\Theta_k}(M)$ is negative if $k$ is even and positive if $k$ is odd.
\end{conjecture}

Our evidence for this conjecture takes three forms.
\begin{enumerate}
\item We prove the conjecture in dimension $2n=4$, i.e., we prove that $b_{\Theta_2}(M)$ is always negative.
\item We give an ``experimental proof'' in higher dimensions.
\item We verify the conjecture for all currently known hyperk{\"a}hler manifolds in dimensions six and eight.
\end{enumerate}

Using results of Salamon~\cite{salamon96} and Verbitsky~\cite{verbitsky96}, Guan~\cite{guan01} bounded the Betti numbers of four-dimensional hyperk{\"a}hler manifolds. The Rozansky-Witten invariants can be computed in terms of the Betti numbers and this suffices to show that $b_{\Theta_2}(M)<0$.

By ``experimental proof'' we mean the following. We generate a random multi-array of complex numbers with the same symmetries as the curvature tensor of a hyperk{\"a}hler manifold. We take copies of this multi-array and contract indices according to the trivalent graph $\Theta_2$. Repeating this computation hundreds of times, we always get negative real numbers. This suggests that $b_{\Theta^{n-2}\Theta_2}(M)$ is negative because the integrand of this invariant is negative at all points of $M$. We perform similar computations for $\Theta_3$, always getting positive real numbers, and therefore expect that the integrand of $b_{\Theta^{n-3}\Theta_3}(M)$ is positive at all points of $M$.

In dimensions six and eight we have Hilbert schemes of points on K3 surfaces and generalized Kummer varieties; their Rozansky-Witten invariants were calculated in Sawon~\cite{sawon00}, and one can readily verify the conjecture is true for these examples. In dimension six there is one additional example due to O'Grady~\cite{ogrady03}, whose Hodge numbers were calculated by Mongardi, Rapagnetta, and Sacc{\`a}~\cite{mrs18}. In six dimensions the Rozansky-Witten invariants can be computed in terms of the Hodge numbers, and this leads to a verification of the conjecture for O'Grady's example too.

As an application, we prove that if $b_{\Theta^{n-2}\Theta_2}(M)<0$ then the second Betti number $b_2$ is bounded. For this result we use a recent result of Jiang~\cite{jiang20} which gives an upper bound on $b_{\Theta^n}(M)$. The degree two cohomology group $\H^2(M,\Z)$ plays a central role in hyperk{\"a}hler geometry. For example, Theorem~4.3 of Huybrechts~\cite{huybrechts03} states that in each dimension there exist only finitely many deformation types of irreducible hyperk{\"a}hler manifolds such that the lattice $(\H^2(M,\Z),\tilde{q})$ is isomorphic to a given one, where $\tilde{q}$ denotes a certain positive integral multiple of the Beauville-Bogomolov quadratic form $q$. Bounding $b_2$ therefore represents a first step toward proving finiteness of hyperk{\"a}hler manifolds. Note however that $q$ and $\tilde{q}$ are not unimodular in general, so even for fixed $b_2$ infinitely many lattices could arise. A different approach to bounding $b_2$ is to analyze the Looijenga-Lunts-Verbitsky decomposition of the cohomology of the hyperk{\"a}hler manifold. Under a hypothesis on the irreducible representations that appear in this decomposition, the author~\cite{sawon15} and Kim and Laza~\cite{kl20} proved that in six dimensions $b_2\leq 23$ (which would be a sharp bound).

Finally, we compare our results to recent work of Cao and Jiang~\cite{cj20} and Jiang~\cite{jiang20}. They investigate the Todd class of a hyperk{\"a}hler manifold and show that it is quasi-effective, which translates into positivity of certain Rozansky-Witten invariants. Our conjecture appears to be stronger, and it implies their results in some cases.

The author would like to thank Simon Salamon for comments on an earlier draft of this paper and the NSF for support through grants DMS-1206309 and DMS-1555206.

\section{Rozansky-Witten invariants}

Rozansky-Witten invariants were defined in~\cite{rw97}, and an equivalent formulation was given by Kapranov~\cite{kapranov99} (see also Kontsevich~\cite{kontsevich99}). We will briefly describe the latter, which is the approach most closely followed in Sawon~\cite{sawon00} and Hitchin and Sawon~\cite{hs01}. Let $M$ be a compact hyperk{\"a}hler manifold of complex dimension $2n$. We fix a compatible complex structure and regard $M$ as a K{\"a}hler manifold. Using the holomorphic symplectic form $\sigma$ to identify $T$ and $\Omega^1$, we can regard the Riemann curvature tensor as a section in
$$\Omega^{1,1}(M,\mathrm{End}T)=\Omega^{0,1}(M,\Omega^1\otimes\Omega^1\otimes T)\cong\Omega^{0,1}(M,\Omega^1\otimes\Omega^1\otimes\Omega^1).$$
The fact that $M$ is hyperk{\"a}hler implies additional symmetries, so that we get
$$\Phi\in\Omega^{0,1}(M,\mathrm{Sym}^3\Omega^1).$$
Let $\Gamma$ be an oriented trivalent graph with $2n$ vertices. Taking $2n$ copies of $\Phi$ and $3n$ copies of the dual $\tilde{\sigma}$ of the symplectic form, we can contract indices according to $\Gamma$ to get a section in $(\Omega^{0,1})^{\otimes 2n}(M)$. Anti-symmetrizing gives
$$\Gamma(\Phi)\in\Omega^{0,2n}(M).$$

\begin{example}
Denote by $\Theta$ and $\Theta_k$, $k\geq 2$, the trivalent graphs
$$\Theta=\twoVgraph\qquad\qquad\Theta_2=\fourVgraph\qquad\qquad\Theta_3=\sixVgraph\qquad\qquad\Theta_4=\eightVgraphI\qquad\qquad\mathrm{etc.}$$
Then
$$\Theta(\Phi)_{\bar{d}\bar{h}}=\Phi_{abc\bar{d}}\Phi_{efg\bar{h}}\tilde{\sigma}^{ae}\tilde{\sigma}^{bf}\tilde{\sigma}^{cg}\in\Omega^{0,2}(M),$$
and
$$\Theta_2(\Phi)_{\bar{d}\bar{h}\bar{l}\bar{p}}=\Phi_{abc\bar{d}}\Phi_{efg\bar{h}}\Phi_{ijk\bar{l}}\Phi_{mno\bar{p}}\tilde{\sigma}^{ae}\tilde{\sigma}^{bf}\tilde{\sigma}^{im}\tilde{\sigma}^{jn}\tilde{\sigma}^{ck}\tilde{\sigma}^{go}\in\Omega^{0,4}(M).$$
The orientation of $\Gamma$ is defined in a way that takes care of any ambiguity in how to order the terms and their indices.
\end{example}

Finally, we can multiply by $\sigma^n\in\Omega^{2n,0}(M)$ and integrate the resulting $(2n,2n)$-form over $M$.

\begin{definition}
The Rozansky-Witten invariant of the hyperk{\"a}hler manifold $M$ corresponding to the trivalent graph $\Gamma$ is
$$b_{\Gamma}(M)=\frac{1}{(8\pi^2)^nn!}\int_M\Gamma(\Phi)\sigma^n.$$
\end{definition}

We now assume that $M$ is irreducible, i.e., simply connected and with $h^{2,0}=1$. Then $\H^{0,2k}_{\bar{\partial}}(M)$ is generated by the class of $\bar{\sigma}^k$ for $k=1,\ldots,n$. In particular, if $\Gamma$ is a trivalent graph with $2k$ vertices then
$$[\Gamma(\Phi)]=\beta_{\Gamma}[\bar{\sigma}^k]\in\H^{0,2k}_{\bar{\partial}}(M)$$
for some constant $\beta_{\Gamma}$. Hitchin and Sawon~\cite{hs01} proved the following result.

\begin{proposition}
$$\beta_{\Theta}=\frac{1}{2n}\frac{\|R\|^2}{\mathrm{vol}M}$$
where $\|R\|^2$ is the $L^2$-norm of the curvature tensor of $M$ and $\mathrm{vol}M$ is its volume. In particular, $\beta_{\Theta}>0$.
\end{proposition}

We can now calculate the Rozansky-Witten invariant associated to $\Theta^n$, the disjoint union of $n$ copies of the graph $\Theta$. We find
$$b_{\Theta^n}(M)=\frac{1}{(8\pi^2)^nn!}\int_M\Theta(\Phi)^n\sigma^n=\frac{1}{(8\pi^2)^nn!}\beta_{\Theta}^n\int_M\sigma^n\bar{\sigma}^n=\frac{n!}{(4\pi^2n)^n}\frac{\|R\|^{2n}}{(\mathrm{vol}M)^{n-1}}>0.$$
Here we used the fact that
$$\int_M\sigma^n\bar{\sigma}^n=\frac{1}{{2n\choose n}}\int_M(\sigma+\bar{\sigma})^{2n}=2^{2n}(n!)^2\int_M\frac{\omega_J^{2n}}{(2n)!}=2^{2n}(n!)^2\mathrm{vol}M.$$
In Sawon~\cite{sawon99} it was conjectured that this invariant is related to Chern numbers, and this was proved in Hitchin and Sawon~\cite{hs01}.

\begin{thm}
$$b_{\Theta^n}(M)=48^nn!\hat{A}^{1/2}[M]$$
where 
$$\hat{A}^{1/2}=1+\frac{1}{24}c_2+\frac{1}{5760}(7c_2^2-4c_4)+\frac{1}{967680}(31c_2^3-44c_2c_4+16c_6)+\ldots$$
is the characteristic class given by the square root of the $\hat{A}$-polynomial.
\end{thm}

Generalizing $\beta_{\Theta}>0$, we make the following conjecture.

\begin{conjecture}
\label{main}
For all irreducible hyperk{\"a}hler manifolds $M$ of complex dimension $2n$ and for $2\leq k\leq n$ we have
$$\beta_{\Theta_k}\left\{\begin{array}{ccl} < & 0 & \mbox{if }k\mbox{ is even,} \\
 > & 0 & \mbox{if }k\mbox{ is odd.} \end{array}\right.$$
Equivalently
 $$b_{\Theta^{n-k}\Theta_k}(M)\left\{\begin{array}{ccl} < & 0 & \mbox{if }k\mbox{ is even,} \\
 > & 0 & \mbox{if }k\mbox{ is odd.} \end{array}\right.$$
 \end{conjecture}

The second statement is equivalent to the first because
$$b_{\Theta^{n-k}\Theta_k}(M)=\frac{1}{(8\pi^2)^nn!}\beta_{\Theta}^{n-k}\beta_{\Theta_k}\int_M\sigma^n\bar{\sigma}^n=\frac{n!}{(2\pi^2)^n}\beta_{\Theta}^{n-k}\beta_{\Theta_k}\mathrm{vol}M$$
and $\beta_{\Theta}>0$. For $k=2$ we have the following result, proved on pages 67 to 72 of Sawon~\cite{sawon00}.

\begin{thm}
The invariant $b_{\Theta^{n-2}\Theta_2}(M)$ is a linear combination of Chern numbers, which can be computed explicitly in any dimension. In dimensions four and six we have
$$b_{\Theta_2}(M)=-\frac{4}{5}(c_2^2-2c_4)[M]$$
and
$$b_{\Theta\Theta_2}(M)=-\frac{8}{35}(11c_2^3-26c_2c_4+12c_6)[M].$$
\end{thm}

Theorem~17 of Sawon~\cite{sawon00} also gives the following.

\begin{thm}
For $k>2$ and $M$ an irreducible hyperk{\"a}hler manifold of dimension $2n$, the invariant $b_{\Theta^{n-k}\Theta_k}(M)$ is a rational function of Chern numbers.
\end{thm}

\begin{remark}
Up to dimension six all Rozansky-Witten invariants are linear combinations of Chern numbers; in particular, this is true for $b_{\Theta_3}(M)$. However, this is no longer true for a general Rozansky-Witten invariant in higher dimensions. For example, Theorem~22 of Sawon~\cite{sawon00} (summarized in Section~5 of Sawon~\cite{sawon01}) says that in dimension eight $b_{\Theta_2^2}(M)$ is not a linear combination of Chern numbers, and it follows that nor is $b_{\Theta\Theta_3}(M)$.
\end{remark}

We end this section by briefly recalling Rozansky and Witten's original formulation, which will be used in the calculations of Section~4. If we denote by $TM$ the real tangent bundle of a hyperk{\"a}hler manifold $M$, then there is a decomposition
$$TM\otimes_{\mathbb{R}}\mathbb{C}=V\otimes S$$
into the tensor product of a (complex) rank $2n$ vector bundle $V$ and a trivializable rank $2$ vector bundle $S$. The Levi-Civita connection on $TM\otimes_{\mathbb{R}}\mathbb{C}$ reduces to an $\mathrm{Sp}(n)$-connection on $V$, and the holomorphic symplectic form $\sigma$ and Riemann curvature tensor $\Phi$ are equivalent to a symplectic form $\epsilon$ on $V$ and a section $\Omega$ of $\mathrm{Sym}^4V$. Rozansky and Witten take copies of $\epsilon$ and $\Omega$ and contract indices according to the trivalent graph, then integrate to produce $b_{\Gamma}(M)$, much like before (see equations~3.41 and 3.43 of~\cite{rw97}). For example, for $\Gamma=\Theta$ in dimension $2n=2$ we would take
$$\Omega_{ABCD}\Omega_{EFGH}\epsilon^{AE}\epsilon^{BF}\epsilon^{CG}.$$
Anti-symmetrizing over the free indices $D$ and $H$ is basically equivalent to contracting with the skew form $\epsilon^{DH}$, and thus
$$b_{\Theta}(M)=\frac{1}{(2\pi)^2}\int_M\Omega_{ABCD}\Omega_{EFGH}\epsilon^{AE}\epsilon^{BF}\epsilon^{CG}\epsilon^{DH}d\mathrm{vol}.$$
For $\Gamma=\Theta_2$ in dimension $2n=4$ we would take 
$$\Omega_{ABCD}\Omega_{EFGH}\Omega_{IJKL}\Omega_{MNOP}\epsilon^{AE}\epsilon^{BF}\epsilon^{IM}\epsilon^{JN}\epsilon^{CK}\epsilon^{OG}.$$
Anti-symmetrizing over the free indices $D$, $H$, $L$, and $P$ is basically equivalent to contracting with the totally skew form
$$\epsilon^{DHLP}:=\epsilon^{DH}\epsilon^{LP}+\epsilon^{DL}\epsilon^{PH}+\epsilon^{DP}\epsilon^{HL},$$
and thus
$$b_{\Theta_2}(M)=\frac{1}{(2\pi)^4}\int_M\Omega_{ABCD}\Omega_{EFGH}\Omega_{IJKL}\Omega_{MNOP}\epsilon^{AE}\epsilon^{BF}\epsilon^{IM}\epsilon^{JN}\epsilon^{CK}\epsilon^{OG}\epsilon^{DHLP}d\mathrm{vol}.$$

\begin{remark}
It is interesting that the final contractions with $\epsilon^{DH}$ and $\epsilon^{DHLP}$ mean that we are really using a $4$-valent graph to contract indices. For $b_{\Theta}(M)$ this would be the $4$-valent graph
$$\twoWgraph,$$
while for $b_{\Theta_2}(M)$ it would be the sum of the $4$-valent graphs
$$\fourWgraphI+\fourWgraphII+\fourWgraphIII.$$
These $4$-valent graphs are really encoding the invariant theory of $\mathrm{Sp}(n)$-representations. The bundle $V$ corresponds to the standard representation $\mathbb{C}^{2n}$ of $\mathrm{Sp}(n)$, whereas $\Omega$ is a section of $\mathrm{Sym}^4V$ which corresponds to $\mathrm{Sym}^4\mathbb{C}^{2n}$. If we decompose $\mathrm{Sym}^2(\mathrm{Sym}^4\mathbb{C}^{2n})$ into irreducible $\mathrm{Sp}(n)$-representations, we find the trivial representation appears with multiplicity one. Indeed, we can identify $\mathbb{C}^{2n}$ with its dual using the symplectic form, and then the invariant subspace of $\mathrm{Sym}^2(\mathrm{Sym}^4\mathbb{C}^{2n})$ is
$$\mathrm{Sym}^2(\mathrm{Sym}^4\mathbb{C}^{2n})^{\mathrm{Sp}(n)}\cong \mathrm{Hom}_{\mathrm{Sp}(n)}(\mathrm{Sym}^4\mathbb{C}^{2n},\mathrm{Sym}^4\mathbb{C}^{2n}),$$
which is one-dimensional by Schur's lemma applied to the {\em irreducible\/} representation $\mathrm{Sym}^4\mathbb{C}^{2n}$. This means there is a unique quadratic invariant of $\Omega$, which gives $b_{\Theta}(M)$.

The decomposition of $\mathrm{Sym}^4(\mathrm{Sym}^4\mathbb{C}^{2n})$ contains four copies of the trivial representation\footnote{We use the form interface to LiE at wwwmathlabo.univ-poitiers.fr/$\sim$maavl/LiE/form.html.}, and thus there are four quartic invariants. These corresponding to the graphs
$$\twoWgraph\cup\twoWgraph\qquad\qquad\fourWgraphI\qquad\qquad\fourWgraphII\qquad\qquad\fourWgraphIII.$$
The first gives the Rozansky-Witten invariant $b_{\Theta^2}(M)$, and the sum of the other three gives $b_{\Theta_2}(M)$. In fact, each individual graph gives a way of contracting indices and thus produces a curvature integral, but it is not clear what these individual curvature integrals represent. Although Conjecture~\ref{main} states that $b_{\Theta_2}(M)<0$ in dimension four, and $\beta_{\Theta_2}<0$ in higher dimensions, we do not expect that these individual curvature integrals are all negative. Indeed, computations like those in Section~4 suggest that they are not (more precisely, computations suggest that the three terms in the integrand for $\beta_{\Theta_2}$ are not individually pointwise negative; only their sum is pointwise negative).

There is also a unique cubic invariant of $\Omega$; it corresponds to the graph
$$\threeWgraph$$
and looks like
$$\Omega_{ABCD}\Omega_{EFGH}\Omega_{IJKL}\epsilon^{AE}\epsilon^{BF}\epsilon^{CI}\epsilon^{DJ}\epsilon^{GK}\epsilon^{HL}.$$
Again, it is not clear what the resulting curvature integral represents. For more relations between graphs and the invariant theory of $\mathrm{Sp}(n)$, see Garoufalidis and Nakamura~\cite{gn98}.
\end{remark}

\section{Dimension four}

In this section we prove Conjecture~\ref{main} in dimension four.

\begin{thm}
Let $M$ be an irreducible hyperk{\"a}hler manifold of complex dimension four. Then
$$b_{\Theta_2}(M)<0.$$
\end{thm}

\begin{proof}
In dimension four we can write the Rozansky-Witten invariants in terms of Chern numbers, which can be written in terms of Betti numbers. Then we can apply Guan's bounds on the Betti numbers~\cite{guan01}. Specifically, from Sawon~\cite{sawon00} (or Hitchin and Sawon~\cite{hs01}) we have
$$\begin{array}{rcl}
b_{\Theta^2}(M) & = & \frac{4}{5}(7c_2^2-4c_4)[M] \\
b_{\Theta_2}(M) & = & -\frac{4}{5}(c_2^2-2c_4)[M], \\
\end{array}$$
and from Section~8 of~\cite{sawon08} we have
$$\begin{array}{rcl}
c_2^2[M] & = & 736+4b_2-b_3 \\
c_4[M] & = & 48+12b_2-3b_3. \\
\end{array}$$
Therefore
$$b_{\Theta_2}(M)=-512+16b_2-4b_3.$$
Guan~\cite{guan01} proved that $b_2$ and $b_3$ can take only finitely many values; the maximum value of $b_{\Theta_2}(M)$ occurs when $b_2=23$ and $b_3=0$, and thus
$$b_{\Theta_2}(M)\leq -144.$$
\end{proof}

\begin{remark}
We also have
$$b_{\Theta^2}(M)=3968-16b_2+4b_3,$$
whose minimum value again occurs when $b_2=23$ and $b_3=0$, and thus
$$b_{\Theta^2}(M)\geq 3600.$$
These values are realized by the Hilbert scheme $S^{[2]}$, and the bounds are therefore sharp. The maximum value occurs when $b_2=3$ and $b_3=68$ (not realized by any currently known hyper{\"a}hler manifold), and thus
$$b_{\Theta^2}(M)\leq 4192.$$
In terms of $\hat{A}^{1/2}[M]=\frac{1}{48^22!}b_{\Theta^2}(M)$ these bounds become
$$\frac{25}{32}\leq\hat{A}^{1/2}[M]\leq\frac{131}{144}.$$
\end{remark}

\section{``Experimental proof''}

The Maple code for our ``experimental proof'' of Conjecture~\ref{main} is in Appendix A, and a dictionary for the variables is in the table below.

\begin{table}[ht]
\begin{center}
\begin{tabular}{|l|l|}
  \hline
$E[i,j]$ & $\epsilon^{IJ}$ \\
$EE[i,j,k,l]$ & $\epsilon^{IJKL}:=\epsilon^{IJ}\epsilon^{KL}+\epsilon^{IK}\epsilon^{LJ}+\epsilon^{IL}\epsilon^{JK}$ \\
$EEE[i,j,k,l,m,p]$ & $\epsilon^{IJKLMP}$ \\
$M[i,j,k,l]$ & $\Omega_{IJKL}$ \\
$V[i,j,k,l]$ & $\Psi_{IJKL}:=\Omega_{ABIJ}\Omega_{CDKL}\epsilon^{AC}\epsilon^{BD}$ \\
$Th$ & $\Theta(\Omega)_{JL}\epsilon^{JL}=\Psi_{IJKL}\epsilon^{IK}\epsilon^{JL}=\Omega_{ABIJ}\Omega_{CDKL}\epsilon^{AC}\epsilon^{BD}\epsilon^{IK}\epsilon^{JL}$ \\
$W[i,j,k,l]$ & $\Psi_{AIBJ}\Psi_{CKDL}\epsilon^{AC}\epsilon^{DB}$ \\
$Th2$ & $\Theta_2(\Omega)_{IJKL}\epsilon^{IJKL}=\Psi_{AIBJ}\Psi_{CKDL}\epsilon^{AC}\epsilon^{DB}\epsilon^{IJKL}$ \\
$U[i,j,k,l,m,p]$ & $\Psi_{AIBJ}\Psi_{CKDL}\Psi_{EMFP}\epsilon^{AD}\epsilon^{CF}\epsilon^{EB}$ \\
$Th3$ & $\Theta_3(\Omega)_{IJKLMP}\epsilon^{IJKLMP}=\Psi_{AIBJ}\Psi_{CKDL}\Psi_{EMFP}\epsilon^{AD}\epsilon^{CF}\epsilon^{EB}\epsilon^{IJKLMP}$ \\
 & \\
  \hline
\end{tabular}
\end{center}
\caption{Dictionary for Maple code}
\end{table}

The first step is to specify the dimension (in the given code we select $2n=6$). In a suitable basis we can assume that $\epsilon^{IJ}$ is the standard skew form, so we define $E[i,j]$ accordingly. Similarly, $EE[i,j,k,l]$ and $EEE[i,j,k,l,m,p]$ are defined to be the standard totally antisymmetric forms.

Next we generate a random multi-array of complex numbers to represent the curvature tensor $\Omega_{IJKL}$ of a hyperk{\"a}hler manifold. We first set $R[i,j,k,l]$ to be a totally symmetric form with random complex entries satisfying $-1\leq \mathcal{R}e\leq 1$ and $-1\leq\mathcal{I}m\leq 1$. The curvature tensor $\Omega_{IJKL}$ must satisfy a reality condition, and thus we define $M[i,j,k,l]$ by adding $R[i,j,k,l]$ to its conjugate under the appropriate reality condition.

Next we take copies of $M[i,j,k,l]$ and $E[i,j]$ and contract indices according to a trivalent graph $\Gamma$. There is one intermediate step: we compute $V[i,j,k,l]$, which corresponds to the tensor
$$\Psi_{IJKL}:=\Omega_{ABIJ}\Omega_{CDKL}\epsilon^{AC}\epsilon^{BD}$$
given by contracting according to the unitrivalent graph $\twowheel$. This piece appears in all of our graphs $\Theta$, $\Theta_2$, and $\Theta_3$, etc., and by computing $V[i,j,k,l]$ once and reusing it in later computations we can speed up the algorithm considerably.

The code outputs the values $Th$, $Th2$, and $Th3$. We can iterate this computation hundreds of times and verify that $Th$ and $Th3$ are always positive and $Th2$ is always negative. Of course, we know that $Th$ must be positive because at each point of $M$
$$\Omega_{ABCD}\Omega_{EFGH}\epsilon^{AE}\epsilon^{BF}\epsilon^{CG}\epsilon^{DH},$$
is the $L^2$-norm of the curvature tensor, and likewise $Th$ is the $L^2$-norm of $M[i,j,k,l]$. As observed earlier, this implies that $b_{\Theta}(M)>0$ in two dimensions, and $\beta_{\Theta}>0$ in arbitrary dimension, because up to a universal positive constant $\beta_{\Theta}$ is the integral over $M$ of 
$$\Omega_{ABCD}\Omega_{EFGH}\epsilon^{AE}\epsilon^{BF}\epsilon^{CG}\epsilon^{DH}.$$
Similarly, the computational verification that $Th2$ is always negative suggests that
$$\Omega_{ABCD}\Omega_{EFGH}\Omega_{IJKL}\Omega_{MNOP}\epsilon^{AE}\epsilon^{BF}\epsilon^{IM}\epsilon^{JN}\epsilon^{CK}\epsilon^{OG}$$
should be a pointwise negative function on $M$. This would imply that $b_{\Theta_2}(M)<0$ in four dimensions, and $\beta_{\Theta_2}<0$ in arbitrary dimension. Finally, the verification that $Th3$ is always positive suggests that $b_{\Theta_3}(M)>0$ in six dimensions, and $\beta_{\Theta_3}>0$ in arbitrary dimension.

\begin{remark}
The value of this ``experimental proof'' is that it suggests an approach to a rigorous proof of the conjecture. Namely, it suggests that these particular Rozansky-Witten invariants are positive (respectively, negative) because their integrands are pointwise positive (respectively, negative), and moreover, that this should be the consequence of a purely algebraic computation. If successful, this would mean that we could obtain valuable topological information about hyperk{\"a}hler manifolds from purely algebraic identities.
\end{remark}

\section{Dimensions six and eight}

In this section we verify Conjecture~\ref{main} for all currently known hyperk{\"a}hler manifolds in complex dimensions six and eight. In dimension six these are the Hilbert scheme of points on a K3 surface $S^{[3]}$, the generalized Kummer variety $K_3(A)$ (see Beauville~\cite{beauville83}), and O'Grady's example $\mathrm{OG6}$~\cite{ogrady03}. In dimension eight these are $S^{[4]}$ and $K_4(A)$. The Rozansky-Witten invariants of $S^{[n]}$ and $K_n(A)$ up to dimension eight were calculated in Sawon~\cite{sawon00} (see Appendix~E.1); the results are reproduced below.
$$
\begin{array}{|c|c|c|c|c|}
\hline
2n & & S^{[n]} & K_n(A) & \mathrm{OG6} \\
\hline
2 & b_{\Theta}     & 48           & & \\
4 & b_{\Theta^2}               & 3600         & 3888 & \\
  & b_{\Theta_2}                & -144         & -432 & \\
6 & b_{\Theta^3}               & 373248       & 442368 & \hspace*{5mm}442368\hspace*{5mm} \\
  & b_{\Theta\Theta_2}      & -13824       & -36864 & -36864 \\
  & b_{\Theta_3}               & 512          & 2560 & 3072 \\
8 & b_{\Theta^4}               & 49787136     & 64800000 & \\
  & b_{\Theta^2\Theta_2}   & \hspace*{5mm}-1693440\hspace*{5mm}    & \hspace*{5mm}-4320000\hspace*{5mm} & \\
  & b_{\Theta_2^2}             & 57600        & 288000 & \\
  & b_{\Theta\Theta_3}       & 56448        & 240000 & \\
  & b_{\Theta_4}                & -1824        & -12000 & \\
  & b_{\Xi}             & 348          & -1500 & \\
\phantom{n} & & &&  \\
\hline
\end{array}
$$
The last invariant in dimension eight, $b_{\Xi}$, corresponds to the trivalent graph which is the skeleton of a cube,
$$\Xi=\eightVgraphII.$$
We have also included the invariants of $\mathrm{OG6}$ in this table; let us explain how they are calculated. The Hodge numbers were calculated by Mongardi, Rapagnetta, and Sacc{\`a}~\cite{mrs18}; they are
\begin{center}
\scalebox{0.95}{
$\begin{array}{ccccccccccccc}
 & & & & & & 1 & & & & & & \\
 & & & & & 0 & & 0 & & & & & \\
 & & & & 1 & & 6 & & 1 & & & & \\
 & & & 0 & & 0 & & 0 & & 0 & & & \\
 & & 1 & & 12 & & 173 & & 12 & & 1 & & \\
 & 0 & & 0 & & 0 & & 0 & & 0 & & 0 & \\
1 & & 6 & & 173 & & 1144 & & 173 & & 6 & & 1 \\
 & 0 & & 0 & & 0 & & 0 & & 0 & & 0 & \\
 & & 1 & & 12 & & 173 & & 12 & & 1 & & \\
 & & & 0 & & 0 & & 0 & & 0 & & & \\
 & & & & 1 & & 6 & & 1 & & & & \\
 & & & & & 0 & & 0 & & & & & \\
 & & & & & & 1 & & & & & & \\
 \end{array}.$
}
\end{center}
From this we can compute the Hirzebruch $\chi_y$-genus
$$\chi_y=\sum_{p,q=0}^6(-1)^qh^{p,q}y^p=4-24y+348y^2-1168y^3+348y^4-24y^5+4y^6.$$ 
By Riemann-Roch, the coefficients $\chi^p=\sum_{q=0}^6(-1)^qh^{p,q}$ can be written in terms of Chern numbers, and these relations can be inverted to give
$$\begin{array}{rcccl}
c_2^3[M] & = & 7272\chi^0-184\chi^1-8\chi^2 & = & 30720 \\ 
c_2c_4[M] & = & 1368\chi^0-208\chi^1-8\chi^2 & = & 7680 \\ 
c_6[M] & = & 36\chi^0-16\chi^1+4\chi^2 & = & 1920 \\ 
\end{array}$$
(see Appendix~B of~\cite{sawon00}). Finally, in dimension six we can compute all Rozansky-Witten invariants in terms of Chern numbers, and we find
$$\begin{array}{rcccl}
b_{\Theta^3}(M) & = & \frac{24}{35}(31c_2^3-44c_2c_4+16c_6)[M] & = & 442368 \\ 
b_{\Theta\Theta_2}(M) & = & -\frac{8}{35}(11c_2^3-26c_2c_4+12c_6)[M] & = & -36864 \\ 
b_{\Theta_3}(M) & = & \frac{8}{35}(c_2^3-3c_2c_4+3c_6)[M] & = & 3072. \\ 
\end{array}$$

\begin{remark}
It is perhaps surprising that some of the Rozansky-Witten invariants of $K_3(A)$ and $\mathrm{OG6}$ are the same. Here is a possible explanation for this. In dimension ten, there is a relation between $S^{[5]}$ and O'Grady's example $\mathrm{OG10}$: they can be deformed to Lagrangian fibrations that are locally isomorphic away from singular fibres. More precisely, on the complement of the singular fibres one fibration is a torsor over the other (see Sawon~\cite{sawon04}, or de Cataldo, Rapagnetta, and Sacc{\`a}~\cite{drs21} where this relation was used to calculate the Hodge numbers of $\mathrm{OG10}$). In particular, the two fibrations have the same discriminant locus and their fibres have the same polarization type. By Theorem~2 of Sawon~\cite{sawon08}, $\sqrt{\hat{A}}[M]$ (equivalently, $b_{\Theta^n}(M)$) can be calculated from the degree of the discriminant locus and the polarization type, and therefore $b_{\Theta^5}(S^{[5]})=b_{\Theta^5}(\mathrm{OG10})$.

We expect a similar relation between $K_3(A)$ and $\mathrm{OG6}$, i.e., after deforming to Lagrangian fibrations, one is a (compactified) torsor over the other. This would explain why $b_{\Theta^3}(K_3(A))=b_{\Theta^3}(\mathrm{OG6})$. The relation $b_{\Theta\Theta_2}(M)=2^{12}3^3-\frac{1}{3}b_{\Theta^3}(M)$ (proved in the next section) then gives $b_{\Theta\Theta_2}(K_3(A))=b_{\Theta\Theta_2}(\mathrm{OG6})$ too.
\end{remark}

\section{Applications}

Our first goal in this section is to prove the following inequality involving Rozansky-Witten invariants and the second Betti number. This generalizes Lemma~3 of Guan~\cite{guan01} to arbitrary dimension, and was known to Guan (see the Remark following Lemma~3 of~\cite{guan01}) but never published.

\begin{thm}
\label{b2inequality}
Let $M$ be an irreducible hyperk{\"a}hler manifold of complex dimension $2n$. Then
$$b_{\Theta^n}(M)\geq
-(b_2+2n-2)b_{\Theta^{n-2}\Theta_2}(M).$$
\end{thm}

In this section we will denote the Beauville-Bogomolov quadratic form of $M$ by $Q\in\mathrm{Sym}^2\H^2(M,\Z)^*$, and denote its dual by $q\in\mathrm{Sym}^2\H^2(M,\Q)\subset\H^4(M,\Q)$.

\begin{lemma}
\label{inequality}
Let $u$ be an arbitrary element of $\H^2(M,\C)$. Then
$$(2n-1)(b_{\Theta^n}(M)+2b_{\Theta^{n-2}\Theta_2}(M))\left(\int_Mqu^{2n-2}\right)^2\geq
(2n-3)b_{\Theta^n}(M)\left(\int_Mq^2u^{2n-4}\right)\left(\int_Mu^{2n}\right).$$
\end{lemma}

\begin{proof}
The second Chern class $c_2$ of $M$ can be written as $p+p^{\perp}$, with $p\in\mathrm{Sym}^2\H^2(M,\Q)$ and $p^{\perp}$ in the primitive cohomology $\H^4_{\mathrm{prim}}(M,\Q)$. Up to scale, $q$ is the unique element of $\mathrm{Sym}^2\H^2(M,\Q)$ that remains of type $(2,2)$ with respect to all complex structures on $M$. Since $c_2$ also remains of type $(2,2)$ with respect to all complex structures, its first component $p$ must be a multiple $\mu$ of $q$, i.e.,
$$c_2=\mu q+p^{\perp}.$$
The constant $\mu$ can be determined by multiplying by $\sigma^{n-1}\bar{\sigma}^{n-1}$ and integrating
$$\int_Mc_2\sigma^{n-1}\bar{\sigma}^{n-1}=\mu\int_Mq\sigma^{n-1}\bar{\sigma}^{n-1}.$$
Moreover
$$c_2^2=\mu^2q^2+2\mu qp^{\perp}+(p^{\perp})^2.$$
Multiplying by $\sigma^{n-2}\bar{\sigma}^{n-2}$ and integrating yields
\begin{eqnarray*}
\int_Mc_2^2\sigma^{n-2}\bar{\sigma}^{n-2} & = &
\mu^2\int_Mq^2\sigma^{n-2}\bar{\sigma}^{n-2}+0+\int_M(p^{\perp})^2\sigma^{n-2}\bar{\sigma}^{n-2}
\\
 & \geq & \mu^2\int_Mq^2\sigma^{n-2}\bar{\sigma}^{n-2}. \\
\end{eqnarray*}
On the first line we have used the fact that $q\sigma^{n-2}\bar{\sigma}^{n-2}\in\mathrm{Sym}^{2n-2}\H^2(M,\Q)$ and  $p^{\perp}\in\H^4_{\mathrm{prim}}(M,\Q)$ are orthogonal, and on the second line we have used the Hodge-Riemann bilinear relations. Substituting the formula for $\mu$ into the inequality yields
$$\left(\int_Mc_2^2\sigma^{n-2}\bar{\sigma}^{n-2}\right)\left(\int_Mq\sigma^{n-1}\bar{\sigma}^{n-1}\right)^2\geq\left(\int_Mc_2\sigma^{n-1}\bar{\sigma}^{n-1}\right)^2\left(\int_Mq^2\sigma^{n-2}\bar{\sigma}^{n-2}\right).$$
The terms involving $c_2$ can be rewritten as Rozansky-Witten invariants. Specifically, Equation~8 of Hitchin and Sawon~\cite{hs01} states that
$$\beta_{\Theta}=\frac{16\pi^2n\int_Mc_2\sigma^{n-1}\bar{\sigma}^{n-1}}{\int_M\sigma^n\bar{\sigma}^n}.$$
Similarly, using Chern-Weil theory (see Section~4 of~\cite{hs01} or Chapter~2 of Sawon~\cite{sawon00}) we find that
$$\beta_{\Theta}^2+2\beta_{\Theta_2}=\frac{(16\pi^2)^2n(n-1)\int_Mc_2^2\sigma^{n-2}\bar{\sigma}^{n-2}}{\int_M\sigma^n\bar{\sigma}^n}.$$
Substituting these into the inequality and simplifying gives
$$n(\beta_{\Theta}^2+2\beta_{\Theta_2})\left(\int_Mq\sigma^{n-1}\bar{\sigma}^{n-1}\right)^2\geq (n-1)\beta_{\Theta}^2\left(\int_Mq^2\sigma^{n-2}\bar{\sigma}^{n-2}\right)\left(\int_M\sigma^n\bar{\sigma}^n\right).$$
Multiplying both sides by $\frac{1}{(8\pi^2)^nn!}\beta_{\Theta}^{n-2}\int\sigma^n\bar{\sigma}^n>0$ converts the terms involving $\beta$s into genuine Rozansky-Witten invariants,
$$n(b_{\Theta^n}(M)+2b_{\Theta^{n-2}\Theta_2}(M))\left(\int_Mq\sigma^{n-1}\bar{\sigma}^{n-1}\right)^2\geq (n-1)b_{\Theta^n}(M)\left(\int_Mq^2\sigma^{n-2}\bar{\sigma}^{n-2}\right)\left(\int_M\sigma^n\bar{\sigma}^n\right).$$
The terms involving $\sigma$ and $\bar{\sigma}$ can be rewritten so that they only depend on the combination $u:=\sigma+\bar{\sigma}$; for example
$$\int_Mq\sigma^{n-1}\bar{\sigma}^{n-1}=\frac{\int_Mq(\sigma+\bar{\sigma})^{2n-2}}
{{2n-2 \choose n-1}}.$$
After simplifying, this gives
$$(2n-1)(b_{\Theta^n}(M)+2b_{\Theta^{n-2}\Theta_2}(M))\left(\int_Mqu^{2n-2}\right)^2
\geq (2n-3)b_{\Theta^n}(M)\left(\int_Mq^2u^{2n-4}\right)\left(\int_Mu^{2n}\right).$$
Since the resulting inequality will remain true under a deformation of the complex structure, it must be true for an arbitrary element $u\in\H^2(M,\C)$.
\end{proof}

The proof of Theorem~\ref{b2inequality} now proceeds by substituting in elements $u\in\H^2(M,\C)$ that lead to the second Betti number $b_2$ appearing. 

\begin{proof}
Let $\{e_1,\ldots,e_{b_2}\}$ be an orthonormal basis for $\H^2(M,\C)$ with respect to the Beauville-Bogomolov form $Q$; then 
$$q=\sum_{i=1}^{b_2}e_i^2.$$
The Fujiki relation~\cite{fujiki87} says that there is a constant $\lambda$ depending only on $M$ such that
$$\int_Mu^{2n}=\lambda Q(u)^n$$
for all $u\in\H^2(M,\C)$. In particular, for distinct $i$, $j$, and $k$ we have
$$\int_M(e_i+te_j+se_k)^{2n}=\lambda Q(e_i+te_j+se_k)^n =\lambda(1+t^2+s^2)^n.$$
By comparing the constant, $t^{2n-2}$, $t^{2n-4}$, and $t^2s^{2n-4}$ terms we find that
$$\int_Me_i^{2n}=\lambda,\qquad\qquad{2n \choose 2}\int_Me_i^2e_j^{2n-2}=\lambda {n \choose 1},\qquad\qquad{2n \choose 4}\int_Me_i^4e_j^{2n-4}=\lambda {n \choose 2},$$
and
$${2n \choose 2,2,2n-4}\int_Me_i^2e_j^2e_k^{2n-4}=\lambda {n \choose 1,1,n-2}.$$
Now we substitute $u=\sum e_i$ into the terms in the inequality of Lemma~\ref{inequality}. Firstly,
$$\int_M\left(\sum e_i\right)^{2n}=\lambda Q\left(\sum e_i\right)^n=\lambda b_2^n,$$
Next, because $q$ remains of type $(2,2)$ under deformations of the complex structure, Corollary~23.17 of Huybrechts' notes in~\cite{ghj02} implies that there exists a constant $\lambda_q$ such that
$$\int_Mqu^{2n-2}=\lambda_qQ(u)^{n-1}$$
for all $u\in\H^2(M,\C)$. To find $\lambda_q$ we substitute $u=e_j$, which gives
\begin{eqnarray*}
\lambda_q & = & \lambda_qQ(e_j)^{n-1} \\
 & = & \int_M\left(\sum e_i^2\right)e_j^{2n-2} \\
 & = & \int_M e_j^{2n}+\sum_{i\neq j}\int_Me_i^2e_j^{2n-2} \\
 & = & \lambda+(b_2-1)\frac{\lambda {n \choose 1}}{{2n \choose 2}} \\
 & = & \lambda\left(\frac{b_2+2n-2}{2n-1}\right). \\
\end{eqnarray*}
Now putting $u=\sum e_i$ gives
$$\int_Mq\left(\sum e_i\right)^{2n-2}=\lambda_qQ\left(\sum e_i\right)^{n-1}=\lambda\left(\frac{b_2+2n-2}{2n-1}\right)b_2^{n-1}.$$
Similarly, $q^2$ remains of type $(4,4)$ under deformations of the complex structure, so there exists a constant $\lambda_{q^2}$ such that
$$\int_Mq^2u^{2n-4}=\lambda_{q^2}Q(u)^{n-2}$$
for all $u\in\H^2(M,\C)$. To find $\lambda_{q^2}$ we substitute $u=e_k$,
which gives
\begin{eqnarray*}
\lambda_{q^2} & = & \lambda_{q^2}Q(e_k)^{n-2} \\
 & = & \int_M\left(\sum e_i^4+\sum_{i\neq j}e_i^2e_j^2\right)e_k^{2n-4} \\
 & = & \int_M e_k^{2n}+\sum_{i\neq k}\int_Me_i^4e_k^{2n-4}+\sum_{i,j,k\mbox{ \footnotesize{distinct}}}\int_Me_i^2e_j^2e_k^{2n-4}+\sum_{i\neq k}\int_Me_i^2e_k^{2n-2}+\sum_{j\neq k}\int_Me_j^2e_k^{2n-2} \\
 & = & \lambda+(b_2-1)\frac{\lambda {n \choose 2}}{{2n \choose
 4}}+(b_2-1)(b_2-2)\frac{\lambda {n \choose 1,1,n-2}}{{2n \choose
 2,2,2n-4}}+2(b_2-1)\frac{\lambda {n \choose 1}}{{2n \choose
 2}} \\
 & = & \lambda\frac{(b_2+2n-2)(b_2+2n-4)}{(2n-1)(2n-3)}. \\
\end{eqnarray*}
Now putting $u=\sum e_i$ gives
$$\int_Mq^2\left(\sum e_i\right)^{2n-4}=\lambda_{q^2}Q\left(\sum e_i\right)^{n-2}=\lambda\frac{(b_2+2n-2)(b_2+2n-4)}{(2n-1)(2n-3)}b_2^{n-2}.$$
Substituting everything into the inequality of Lemma~\ref{inequality} and simplifying yields the inequality
$$(b_{\Theta^n}(M)+2b_{\Theta^{n-2}\Theta_2}(M))(b_2+2n-2)\geq b_{\Theta^n}(M)(b_2+2n-4),$$
which can be rearranged to give
$$b_{\Theta^n}(M)\geq -(b_2+2n-2)b_{\Theta^{n-2}\Theta_2}(M).$$
\end{proof}

Our main conjecture, Conjecture~\ref{main}, asserts that $b_{\Theta^{n-2}\Theta_2}(M)<0$, which would give the following.

\begin{corollary}
If $b_{\Theta^{n-2}\Theta_2}(M)<0$ then the second Betti number $b_2$ is bounded above.
\end{corollary}

\begin{proof}
Because $b_{\Theta^{n-2}\Theta_2}(M)$ is a (rational) linear combination of Chern numbers, if it is negative then it must be $\leq -1/C_n$ where $C_n$ is a positive integer depending only on $n$. On the other hand, Jiang proved that $\hat{A}^{1/2}[M]$ is strictly less than $1$ (Corollary~5.5 of~\cite{jiang20}), and this implies that
$$b_{\Theta^n}(M)=48^nn!\hat{A}^{1/2}[M]<48^nn!.$$
Theorem~\ref{b2inequality} then gives
$$48^nn!>b_{\Theta^n}(M)\geq -(b_2+2n-2)b_{\Theta^{n-2}\Theta_2}(M)\geq\frac{b_2+2n-2}{C_n},$$
so $b_2$ is bounded above by $48^nn!C_n-2n+2$.
\end{proof}

\begin{example}
In dimension four we already have the sharp bound $b_2\leq 23$, due to Guan~\cite{guan01}. In dimension six we have
$$\hat{A}^{1/2}[M]=\frac{1}{967680}(31c_2^3-44c_2c_4+16c_6)[M],$$
so $\hat{A}^{1/2}[M]<1$ implies
$$b_{\Theta^3}(M)\leq 48^33!\frac{967679}{967680}=\frac{23224296}{35}.$$
In addition
$$b_{\Theta\Theta_2}(M)=-\frac{8}{35}(11c_2^3-26c_2c_4+12c_6)[M],$$
so $b_{\Theta\Theta_2}(M)<0$ would imply $b_{\Theta\Theta_2}(M)\leq -\frac{8}{35}$. Therefore
$$\frac{23224296}{35}\geq b_{\Theta^n}(M)\geq -(b_2+4)b_{\Theta^{n-2}\Theta_2}(M)\geq (b_2+4)\frac{8}{35},$$
which simplifies to give
$$2903033\geq b_2.$$
\end{example}

In fact, we can improve on this. The first step is the following surprising relation.

\begin{lemma}
\label{six}
Let $M$ be an irreducible hyperk{\"a}hler manifold of complex dimension six. Then
$$b_{\Theta^3}(M)+3b_{\Theta\Theta_2}(M)=2^{10}3^4\chi(\O_M)=2^{12}3^4.$$
\end{lemma}

\begin{proof}
Writing the left-hand side as a linear combination of Chern numbers, we find
$$b_{\Theta^3}(M)+3b_{\Theta\Theta_2}(M)=\frac{48}{35}(10c_2^3-9c_2c_4+2c_6)[M].$$
By the Riemann-Roch theorem, $\chi(\O_M)$ can also be written as a linear combination of Chern
numbers, and a calculation reveals agreement with the above expression up to the constant $2^{10}3^4$. Finally, one notes that $\chi(\O_M)=4$ for an irreducible hyperk{\"a}hler manifold of complex dimension six.
\end{proof}

Our Conjecture~\ref{main} immediately yields a lower bound for $b_{\Theta^3}(M)$.

\begin{corollary}
If $b_{\Theta\Theta_2}(M)<0$ then $b_{\Theta^3}(M)>2^{12}3^4$. Equivalently,
$$\hat{A}^{1/2}[M]=\frac{1}{48^33!}b_{\Theta^3}(M)>\frac{1}{2}.$$
\end{corollary}

Now we can improve the upper bound on the second Betti number $b_2$.

\begin{thm}
If $b_{\Theta\Theta_2}(M)<0$ then
$$b_2\leq 1451519.$$
\end{thm}

\begin{remark}
We expect that this bound must be extremely crude. There are three known examples of irreducible hyperk{\"a}hler manifolds in dimension six, and their second Betti numbers are $7$, $8$, and $23$. The results of Sawon~\cite{sawon15} and Kim and Laza~\cite{kl20} suggest that $23$ is most likely the largest possible value for $b_2$.
\end{remark}

\begin{proof}
Combining Theorem~\ref{b2inequality}, $b_{\Theta\Theta_2}(M)<0$, and Lemma~\ref{six} gives
$$b_2+4\leq\frac{b_{\Theta^3}(M)}{-b_{\Theta\Theta_2}(M)}=\frac{3b_{\Theta^3}(M)}{b_{\Theta^3}(M)-2^{12}3^4}.$$
The right-hand side is a decreasing function of $b_{\Theta^3}(M)$, so it is maximized when $b_{\Theta^3}(M)$ is smallest. Now $b_{\Theta^3}(M)>2^{12}3^4$ and
$$b_{\Theta^3}(M)=\frac{24}{35}(31c_2^3-44c_2c_4+16c_6)[M]\in\frac{24}{35}\Z,$$
therefore $b_{\Theta^3}(M)$ is at least $2^{12}3^4+\frac{24}{35}$, and substituting this values gives
$$b_2+4\leq 1451523.$$
\end{proof}

We also obtain an upper bound for $b_{\Theta^3}(M)$.

\begin{proposition}
If $b_{\Theta\Theta_2}(M)<0$ then $b_{\Theta^3}(M)\leq 2^{10}3^47$. Equivalently,
$$\hat{A}^{1/2}[M]=\frac{1}{48^33!}b_{\Theta^3}(M)\leq\frac{7}{8}.$$
\end{proposition}

\begin{remark}
This would improve on Jiang's bound $\hat{A}^{1/2}[M]<1$ (though note that Jiang's bound holds in all dimensions, not just dimension six).
\end{remark}

\begin{proof}
The second Betti number $b_2$ must be at least three, and therefore the inequality at the start of the last proof gives
$$7\leq\frac{3b_{\Theta^3}(M)}{b_{\Theta^3}(M)-2^{12}3^4}.$$
Rearranging gives $b_{\Theta^3}(M)\leq 2^{10}3^47$.
\end{proof}

\begin{example}
For the Hilbert scheme $S^{[3]}$ of three points on a K3 surface we have $b_{\Theta^3}(M)=373248$ and $b_2=23$. This implies that we have equality in
$$b_2+4\leq\frac{3b_{\Theta^3}(M)}{b_{\Theta^3}(M)-2^{12}3^4}.$$
It follows that in the decomposition of the second Chern class $c_2=p+p^{\perp}$ in the proof of Lemma~\ref{inequality}, we must have $p^{\perp}=0$. In other words, $c_2=p$ lies in the part $\mathrm{Sym}^2\H^2(M,\Q)$ of $\H^4(M,\Q)$ which is generated by $\H^2(M,\Q)$.
\end{example}

\begin{example}
For the generalized Kummer variety $K_3(A)$ and O'Grady's example $\mathrm{OG6}$ we have $b_{\Theta^3}(M)=442368$ and $b_2=7$ and $8$, respectively. This implies that we have equality in
$$b_2+4\leq\frac{3b_{\Theta^3}(M)}{b_{\Theta^3}(M)-2^{12}3^4}$$
for $\mathrm{OG6}$, but strict inequality for $K_3(A)$. Thus for $\mathrm{OG6}$, $p^{\perp}=0$ and $c_2$ lies in $\mathrm{Sym}^2\H^2(M,\Q)$, but for $K_3(A)$ this is no longer true.
\end{example}

\section{Comparisons to other work}

Cao and Jiang proposed (Conjecture~3.6 in~\cite{cj20}), and Jiang later proved (Corollary~5.3 in~\cite{jiang20}), that Todd classes of hyperk{\"a}hler manifolds are quasi-effective, in the following sense.

\begin{thm}
Let $X$ be a hyperk{\"a}hler manifold of complex dimension $2n$ and $L$ a nef and big line bundle on $X$. Then $\int_M\mathrm{td}_{2n-2i}L^{2i}>0$ for all $0\leq i\leq n$.
\end{thm}

This is equivalent to proving $\int_M\mathrm{td}_{2n-2i}\sigma^i\bar{\sigma}^i>0$, or in other words, proving that certain Rozansky-Witten invariants are positive. For example, consider the quasi-effectivity of $\mathrm{td}_4$, which is equivalent to
$$\int_M\mathrm{td}_4\sigma^{n-2}\bar{\sigma}^{n-2}=\int_M\frac{1}{720}(3c_2^2-c_4)\sigma^{n-2}\bar{\sigma}^{n-2}>0.$$
This case was already proved by Cao and Jiang, Theorem~3.2 in~\cite{cj20}, by writing
\begin{eqnarray*}
\mathrm{td}_4 & = & (\mathrm{td}^{1/2})_0(\mathrm{td}^{1/2})_4+(\mathrm{td}^{1/2})_2(\mathrm{td}^{1/2})_2+(\mathrm{td}^{1/2})_4(\mathrm{td}^{1/2})_0 \\
 & = & 2(\mathrm{td}^{1/2})_4+ (\mathrm{td}^{1/2})_2^2 \\
 & = & \frac{1}{2880}(7c_2^2-4c_4)+\frac{1}{576}c_2^2,
\end{eqnarray*}
and then showing that both
$$\int_M(7c_2^2-4c_4)\sigma^{n-2}\bar{\sigma}^{n-2}\qquad\qquad\mbox{and}\qquad\qquad\int_Mc_2^2\sigma^{n-2}\bar{\sigma}^{n-2}$$
are positive.

Now as already observed in the proof of Lemma~\ref{inequality}, using Chern-Weil theory (see Section~4 of Hitchin and Sawon~\cite{hs01} or Chapter~2 of Sawon~\cite{sawon00}) we can show that
$$\beta_{\Theta}^2+2\beta_{\Theta_2}=\frac{(8\pi^2)^2n(n-1)\int_M(s_2^2)\sigma^{n-2}\bar{\sigma}^{n-2}}{\int_M\sigma^n\bar{\sigma}^n}=\frac{(8\pi^2)^2n(n-1)\int_M4c_2^2\sigma^{n-2}\bar{\sigma}^{n-2}}{\int_M\sigma^n\bar{\sigma}^n}.$$
In the same way we find that
$$\frac{5}{2}\beta_{\Theta_2}=\frac{(8\pi^2)^2n(n-1)\int_M(-s_4)\sigma^{n-2}\bar{\sigma}^{n-2}}{\int_M\sigma^n\bar{\sigma}^n}=-\frac{(8\pi^2)^2n(n-1)\int_M(2c_2^2-4c_4)\sigma^{n-2}\bar{\sigma}^{n-2}}{\int_M\sigma^n\bar{\sigma}^n}.$$
Here $s_2=2!ch_2=-2c_2$ and $s_4=4!ch_4=2c_2^2-4c_4$. The fact that $\int_Mc_2^2\sigma^{n-2}\bar{\sigma}^{n-2}>0$ then implies $\beta_{\Theta}^2+2\beta_{\Theta_2}>0$, but this only gives a negative lower bound of $-\beta_{\Theta}^2/2$ for $\beta_{\Theta_2}$.

Finally, we observe that our Conjecture~\ref{main} that $\beta_{\Theta_2}<0$ is equivalent to
$$\int_M(2c_2^2-4c_4)\sigma^{n-2}\bar{\sigma}^{n-2}>0,$$
which is strictly stronger than the known result
$$\int_M(7c_2^2-4c_4)\sigma^{n-2}\bar{\sigma}^{n-2}>0$$
from above. Indeed the latter follows from the former by writing
$$\int_M(7c_2^2-4c_4)\sigma^{n-2}\bar{\sigma}^{n-2}=\int_M5c_2^2\sigma^{n-2}\bar{\sigma}^{n-2}+\int_M(2c_2^2-4c_4)\sigma^{n-2}\bar{\sigma}^{n-2}.$$






\newpage
\appendix
\section{Maple code}

\vspace*{-15mm}\hspace*{-5mm}
\includegraphics[scale=0.75,page=1]{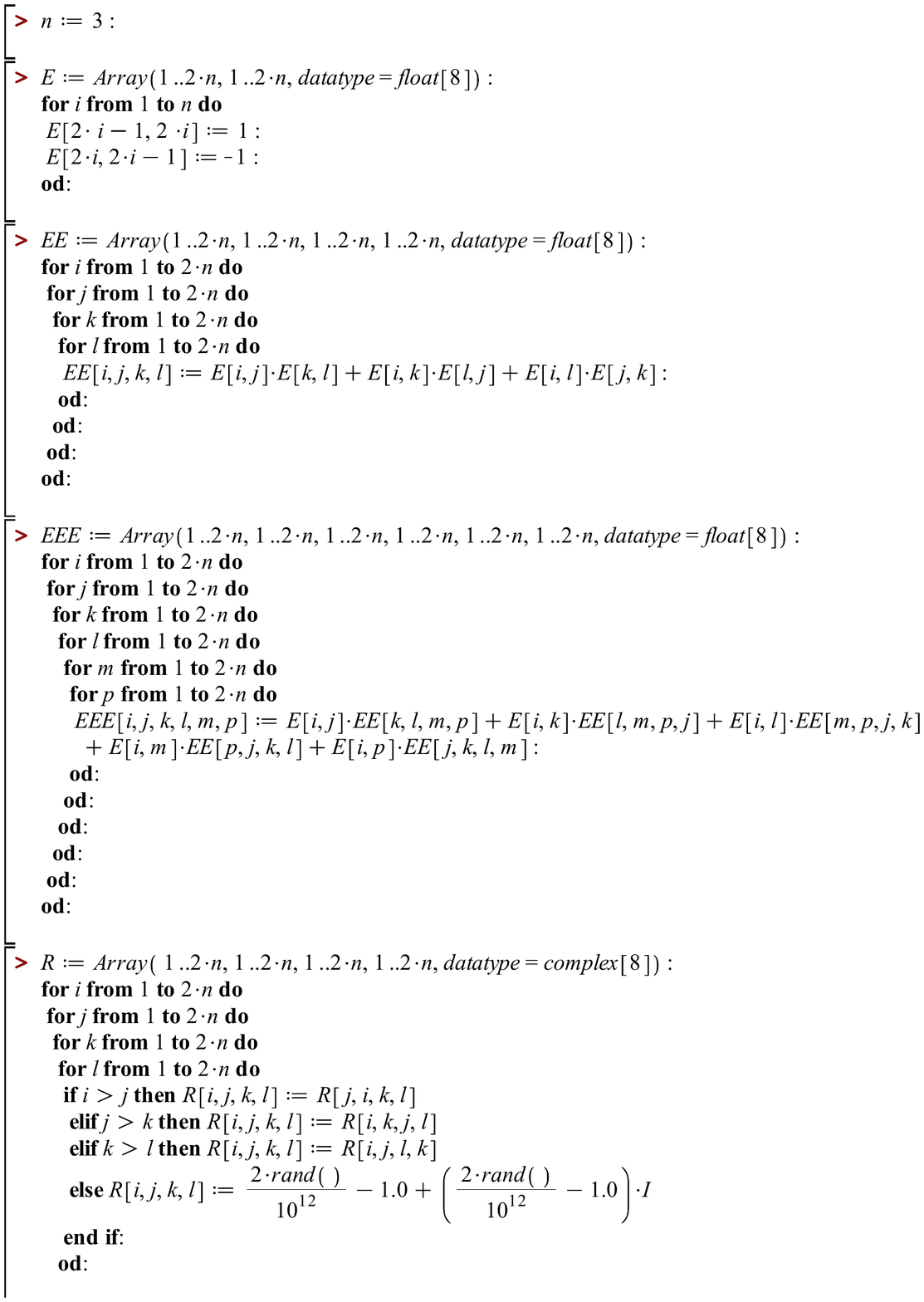}

\newpage
\vspace*{-15mm}\hspace*{-10mm}
\includegraphics[scale=0.80,page=2]{Betti_bound_codes3.pdf}

\newpage
\vspace*{-15mm}\hspace*{-10mm}
\includegraphics[scale=0.80,page=3]{Betti_bound_codes3.pdf}


\begin{flushleft}
Department of Mathematics\hfill sawon@email.unc.edu\\
University of North Carolina\hfill sawon.web.unc.edu\\
Chapel Hill NC 27599-3250\\
USA\\
\end{flushleft}

\end{document}